\input ./style/arxiv-general.cfg
\documentclass[aap,MSNbibl,seceqn,dvips]{arximspdf}
\makeatletter
   \@ifpackageloaded{graphicx}{}{\usepackage{graphicx}}
\makeatother

\doi{10.1214/14-AAP1082}
\volume{25}
\issue{6}
\pubyear{2015}
\firstpage{3571}
\lastpage{3591}
\docsubty{FLA}

\makeatletter
\newcommand{\rrvert}{\vert}
\newcommand{\llvert}{\vert}

\newtheorem{theorem}{Theorem}[section]
\newtheorem{lemma}[theorem]{Lemma}

\newproclaim{example}[theorem]{Example}
\newproclaim{definition}[theorem]{Definition}
\newproclaim{problem}[theorem]{Problem}
\newproclaim{remark}[theorem]{Remark}

\makeatother

\begin{document}
\begin{frontmatter}

\title{Asymptotic distribution of the maximum interpoint
distance in a sample of random
vectors with a~spherically symmetric distribution}
\runtitle{Maximum interpoint distance}

\begin{aug}
\author[A]{\fnms{Sreenivasa Rao}~\snm{Jammalamadaka}\ead[label=e1]{rao@pstat.ucsb.edu}\ead[label=u1,url]{http://www.pstat.ucsb.edu/faculty/jammalam/}}
\and
\author[B]{\fnms{Svante}~\snm{Janson}\corref{}\thanksref{T1}\ead[label=e2]{svante.janson@math.uu.se}\ead[label=u2,url]{http://www2.math.uu.se/\textasciitilde svante/}}
\runauthor{S.~R. Jammalamadaka and S. Janson}
\thankstext{T1}{Supported in part by the Knut and Alice Wallenberg Foundation.}
\affiliation{University of California and Uppsala University}
\address[A]{Department of Statistics\\
\quad and Applied Probability\\
University of California\\
Santa Barbara, California 93106\\
USA\\
\printead{e1}\\
\printead{u1}}
\address[B]{Department of Mathematics\\
Uppsala University\\
PO Box 480\\
SE-751~06 Uppsala\\
Sweden\\
\printead{e2}\\
\printead{u2}}
\end{aug}

%
\received{\smonth{12} \syear{2012}}
%
\revised{\smonth{11} \syear{2014}}

%
\begin{abstract}
Extreme value theory is part and parcel of any study of order
statistics in
one dimension. Our aim here is to consider such large sample theory for the
maximum distance to the origin, and the related maximum ``interpoint
distance,'' in multidimensions. We show that for a family of spherically
symmetric distributions, these statistics have a Gumbel-type limit,
generalizing several existing results. We also discuss the other two types
of limit laws and suggest some open problems. This work complements
our earlier study on the minimum interpoint distance.
\end{abstract}

%
\end{frontmatter}

\section{Introduction and main results}\label{S:intro}

Let $X_1,X_2,\ldots$ be an independently and identically distributed
(i.i.d.)
sequence of random vectors in $\mathbb R^d$
with a spherically symmetric distribution, where $d\ge2$. (See
Section~\ref{SSd=1}
for some comments on the case $d=1$; otherwise $d\ge2$ will always be assumed.)
We are interested in the maximum interpoint distance
%
\begin{equation}
\label{MM} M^{(2)}_n:=\max_{1\le i< j\le n}|X_i-X_j|,
\end{equation}
where $|\cdot|$ is the usual Euclidean distance.
This has previously been studied by several authors in various special
cases, including
Matthews and Rukhin \cite{MR93} (symmetric normal distribution),
Henze and Klein \cite{HK96} (Kotz distribution),
Appel, Najim and Russo \cite{ANR02} (uniform distribution in a ball),
Appel and Russo \cite{AR06} (uniform distribution on a sphere) and
Mayer and Molchanov \cite{MM07} (e.g., uniform distribution in a ball
or on a sphere).
We provide some general results here
for the case of unbounded random vectors from any spherically symmetric
distribution, which includes the work in \cite{MR93} and \cite{HK96}.

The results for maximum distance can be considered as complementary to
the results for the
minimum interpoint distance; see, for example, Jammalamadaka and Janson
\cite{JJ}.
One important difference is that the minimum distance is typically
achieved by points in the bulk of the distribution, while the maximum
distance is attained by outliers.
This makes the maximum distance less useful for goodness of fit tests, but
might be suitable for detecting outliers. Some
applications are given in Matthews and Rukhin \cite{MR93}.

The maximum pairwise distance $M^{(2)}_n$ is clearly related to the maximum
distance to the origin
%
\begin{equation}
\label{M} M_n:=\max_{1\le i\le n}|X_i|.
\end{equation}
We obviously have $M^{(2)}_n\le2M_n$, and it seems
reasonable to guess that this bound is rather sharp; this would mean that
the maximum distance (\ref{MM}) is attained by two vectors $X_i$ and $X_j$
that have almost maximum lengths and are almost opposite in direction.


For the case $d=1$, it is well known (see, e.g., Leadbetter, Lindgren
and Rootz\'en~\cite{LLR}),
that the asymptotic distribution of $M_n$ after suitable
normalization, may be of one of the three different types (assuming
that the tail of
the distribution of $|X_i|$
is so regular that there is an asymptotic distribution at all). The three
types of limit distributions, called extreme value distributions,
are known as Gumbel, Weibull and Fr\'echet distributions;
they have the distribution functions
%
\begin{eqnarray}
&&\exp \bigl(-e^{-x} \bigr), \qquad -\infty<x<\infty \mbox{ (Gumbel)},\label{gumbel}
\\
&&\exp \bigl(-|x|^\alpha \bigr),\qquad -\infty<x\le0  \mbox{ (Weibull)},
\label{weibull}
\\
&&\exp \bigl(-x^{-\alpha} \bigr),\qquad 0<x<\infty \mbox{ (Fr\'echet)},\label
{frechet}
\end{eqnarray}
where for the last two cases $\alpha$ is a positive parameter.

For the multidimensional situation, we shall focus here mostly on the
Gumbel limit which includes, for example, the important
case of samples from a normal
distribution; we show that under some regularity conditions
$M^{(2)}_n$ in multidimensions also has an asymptotic Gumbel distribution.
The Weibull case (including, e.g., the uniform distribution in a ball)
was considered in \cite{MM07};
in this case the asymptotic distribution of $M^{(2)}_n$ also turns out
to be
Weibull, although with a different parameter.
We have not much to add to their results except to make
a few comments in Section~\ref{SSweibull}.
The Fr\'echet case (e.g., power law tails) is more complicated; there
is a limit distribution for $M^{(2)}_n$ in this case too, but it is
not known explicitly. We explain this difference in Section~\ref{SSfrechet}.
\subsection{Notation}
All unspecified limits are as ${n\to\infty}$. In particular, $x_n\sim
y_n$ means
$x_n/y_n\to1$ as ${n\to\infty}$ (allowing also $x_n=y_n=0$ for some $n$).
Convergence in probability or distribution is denoted by $\stackrel
{{p}}{\longrightarrow}$ and
$\stackrel{{d}}{\longrightarrow}$,
respectively.
We let $x_+:=\max(x,0)$ for $x\in\mathbb R$.

\subsection{Main results}
Our main result is contained in the following theorem, whose proof is
given in Section~\ref{Spf}. We also provide two
special versions of this main result (Theorem~\ref{TC1} and
Theorem~\ref{TC2})
which readily connect to useful applications.

\begin{theorem}\label{T1}
Suppose that $d\ge2$ and that
$X,X_1,X_2,\ldots$ are i.i.d. $\mathbb R^d$-valued random vectors
with a spherically symmetric distribution
such that for some sequences $a_n$ and $b_n$ of positive numbers
with $b_n=o(a_n)$,
%
\begin{equation}
\label{a} \mathbb{P} \bigl(|X|>a_n+tb_n \bigr) =
\frac{1+o(1)}ne^{-t}
\end{equation}
as ${n\to\infty}$,
for all $t=t_n$ with
$|t|\le\frac{d-1}2\log(a_n/b_n)$.
Let
%
\begin{equation}
\label{cd} c_d:=(d-1)2^{d-4}\Gamma(d/2)/\sqrt\pi.
\end{equation}
Then
%
\begin{equation}
\label{t1a} \frac{M^{(2)}_n-2a_n}{b_n} +\frac{d-1}2\log\frac{a_n}{b_n}-\log\log
\frac{a_n}{b_n}-\log c_d \stackrel{{d}} {\longrightarrow}V,
\end{equation}
where $V$ has the Gumbel distribution $\mathbb{P}(V\le x)=e^{-e^{-x}}$.
\end{theorem}

\begin{remark}\label{Rmax}
In particular, since $\log(a_n/b_n)\to\infty$,
we assume that (\ref{a}) holds for every fixed $t$. This is, by a standard
argument (see, e.g., \cite{LLR}),
equivalent to
%
\begin{equation}
\label{max1} \mathbb{P} \bigl((M_n -a_n)
/b_n\le t \bigr) \to e^{-e^{-t}},
\end{equation}
that is,
%
\begin{equation}
\label{max2} \frac{M_n -a_n}{b_n}\stackrel{{d}} {\longrightarrow}V,
\end{equation}
where $V$ has the Gumbel distribution. (This verifies our claim that we are
dealing with the Gumbel case.)
Conversely, if (\ref{max2}) holds, so (\ref{a}) holds for every fixed~$t$,
then necessarily $b_n=o(a_n)$ (as is easily seen by considering large
negative $t$). Thus the assumption $b_n=o(a_n)$ is redundant if we add the
requirement that (\ref{a}) holds for any fixed $t$.

We also note that our assumption is a bit stronger than just assuming
(\ref{max2}), since we require (\ref{a}) also for some $t=t_n\to
\infty$.
First, this restricts the choice of~$b_n$. Indeed, in (\ref{max2}), $b_n$
can be replaced by any $b_n'=b_n(1+o(1))$, but for assumption
(\ref{a}) for our range of $t$ one needs $b_n'=b_n(1+o(1/\log(a_n/b_n)))$.
Actually, the latter condition is also needed for replacing $b_n$ by $b_n'$
in the conclusion (\ref{t1a}), giving some justification to our condition.
Second, condition (\ref{a}) for our range of $t$ is satisfied for a
suitable choice of $a_n$ and $b_n$ in
sufficiently regular instances of (\ref{max2}), such as
the examples in Section~\ref{Sex}, but it does not always hold.
A counterexample is given by $\mathbb{P}(|X|>x)=\exp (-\int_0^x
h(t)\,\mathrm{d}t )$
with a function $h(t)>0$ such that $h(t)\to1$ as $t\to\infty$; this always
satisfies (\ref{a}) for fixed $t$
(with $b_n\sim1$ and some $a_n\sim\log n$),
and thus (\ref{max2}),
but for a suitably slowly oscillating
$h$, for example, $h(x)=1+\sin(x/\log x)/\log\log x$, (\ref{a}) does
not hold
for all $t$ with $|t|\le\frac{1}2\log(a_n/b_n)$,\vadjust{\goodbreak} for any such $a_n$
and $b_n$.
We expect that it is possible to extend Theorem~\ref{T1} to such
cases, with some
modification of (\ref{t1a}), but we have not pursued this and leave it
as an
open problem.
\end{remark}

\begin{remark}\label{Rdiff}
As a corollary we see that typically $2M_n-M^{(2)}_n$ is about
$\frac{d-1}2 b_n \log(a_n/b_n)$; more precisely,
(\ref{t1a}) and (\ref{max2}) imply
%
\begin{equation}
\frac{2M_n-M^{(2)}_n}{b_n\log(a_n/b_n)}\stackrel{ {p}} {\longrightarrow}\frac{d-1}2.
\end{equation}
It can be seen from the proof below that if we order $X_1,\ldots,X_n$ as
$X_{(1)},\ldots,\break X_{(n)}$ with $M_n=|X_{(1)}|\ge\cdots\ge|X_{(n)}|$,
then the probability that $M^{(2)}_n$ is attained by a pair including $X_{(1)}$
tends to 0; the reason is that the other
large vectors $X_{(2)},\ldots$ probably are not almost opposite to $X_{(1)}$.
However, if we consider points $X_i$ such that $|X_i|$ is close
to $M_n$, with a suitable margin, then
there will be many such points, and it is likely that some pair will be
almost opposite. There is a trade-off between what we lose in length and
what we gain in angle, and the proof of the theorem is based on finding
the right balance.
\end{remark}

We now give two special versions of the main result that are more
conveniently stated, and are most
likely to be useful in applications. The proofs of these two theorems are
given in Section~\ref{Spf2}.

\begin{theorem}\label{TC1}
Suppose that $d\ge2$ and
that $X,X_1,X_2,\ldots$ are i.i.d. $\mathbb R^d$-valued random vectors
with a spherically symmetric distribution
such that
%
\begin{equation}
\label{g} \mathbb{P}\bigl(|X|>x\bigr)=G(x)=e^{-g(x)+o(1)} \qquad\mbox{as } x\to\infty,
\end{equation}
for some twice differentiable function $g(x)$ such that, as ${x\to
\infty}$,
%
\begin{eqnarray}
xg'(x)&\to&\infty, \label{g'}
\\
\frac{g''(x)}{g'(x)^2}\log^2 \bigl(xg'(x) \bigr)&\to&0,
\label{g''}
\end{eqnarray}
and that $a_n$ and $b_n$ are such that, as ${n\to\infty}$, $a_n\to
\infty$ and
%
\begin{eqnarray}
g(a_n)&=&\log n+o(1),\label{an}
\\
b_n&=&\frac{1+o (1/\log(a_ng'(a_n)) )}{g'(a_n)}.\label{bn}
\end{eqnarray}
Then 
(\ref{t1a}) holds.
\end{theorem}

Note that Remark~\ref{Rmax} gives an example of a distribution in the Gumbel
domain of attraction such that (\ref{g}) and (\ref{g'}) hold, but not the
more technical assumption~(\ref{g''}).

\begin{theorem}
\label{TC2}
Suppose that $d\ge2$ and
that $X_1,X_2,\ldots$ are i.i.d. $\mathbb R^d$-valued random vectors
with a spherically symmetric distribution with a density function
$f(\mathbf{x})$ such
that, as $|\mathbf{x}|\to\infty$,
%
\begin{equation}
\label{f} f(\mathbf{x})\sim c |\mathbf{x}|^\alpha e^{-\beta|\mathbf
{x}|^\gamma}
\end{equation}
for some $c,\beta,\gamma>0$ and $\alpha\in\mathbb R$.
Then
\begin{eqnarray*}
&&\bigl(\beta^{1/\gamma}\gamma\log^{1-1/\gamma}n \bigr)\cdot
M^{(2)}_n - \biggl( 2\gamma\log n + \biggl(2
\frac{\alpha+d}{\gamma}-\frac{d+3}2 \biggr)\log\log n
\\
&&\hspace*{123pt}\quad{}+\log\log\log n +\log \bigl(c'_d\beta^{-2(\alpha+d)/\gamma}
\gamma ^{-(d+3)/2}c^2 \bigr) \biggr)\\
&&\qquad \stackrel{{d}} {
\longrightarrow}V,
\end{eqnarray*}
where
%
\begin{equation}
\label{cct} c'_d= \frac{(d-1)2^{d-2}\pi^{d-1/2}}{\Gamma(d/2)},
\end{equation}
and $V$ has the Gumbel distribution. 
\end{theorem}

We give some specific examples in Section~\ref{Sex}, and provide further
comments as well as state some open
problems in Section~\ref{Sfurther}.

\section{Examples}\label{Sex}

\begin{example}
Suppose that $X_i$ has a standard multivariate normal distribution in $R^d$.
The density function is
%
\begin{equation}
f(\mathbf{x})=(2\pi)^{-d/2} e^{-|\mathbf{x}|^2/2},
\end{equation}
which satisfies (\ref{f}) with
$c=(2\pi)^{-d/2}$, $\alpha=0$, $\beta=1/2$ and $\gamma=2$.
Hence Theorem~\ref{TC2} yields, for $d\ge2$,
\begin{eqnarray*}
&&\sqrt{2\log n} M^{(2)}_n - \biggl( 4\log n +
\frac{d-3}{2}\log\log n +\log\log\log n +\log{\frac{(d-1)2^{(d-7)/2}}{\sqrt\pi\Gamma(d/2)}} \biggr)
\\
&&\qquad\stackrel{{d}} {\longrightarrow}V.
\end{eqnarray*}
This was shown by Matthews and Rukhin \cite{MR93} (with a correction
by Henze and Klein \cite{HK96}).
\end{example}

\begin{example}
Henze and Klein \cite{HK96} considered, more generally,
the case when
$X_i$ has a symmetric Kotz-type distribution in $\mathbb R^d$, $d\ge2$,
with density
%
\begin{equation}
f(\mathbf{x})=\frac{\kappa^{d/2+b-1}\Gamma(d/2)}{\pi^{d/2}\Gamma(d/2+b-1)} |\mathbf{x}|^{2(b-1)} e^{-\kappa|\mathbf{x}|^2},
\end{equation}
where $b\in\mathbb R$ and $\kappa>0$.
Theorem~\ref{TC2} applies with
$c=\frac{\kappa^{d/2+b-1}\Gamma(d/2)}{\pi^{d/2}\Gamma(d/2+b-1)}$,
$\alpha=2(b-1)$, $\beta=\kappa$ and $\gamma=2$, and yields
\begin{eqnarray*}
&&{\sqrt{4\kappa\log n}} M^{(2)}_n - \biggl( 4\log n +{
\frac{4b+d-7}2}\log\log n +\log\log\log n
\\
&&\hspace*{128pt}\qquad{}+\log{\frac{(d-1)2^{(d-7)/2}\Gamma(d/2)}{\sqrt\pi\Gamma(d/2+b-1)^2}} \biggr) 
\\
&&\qquad
\stackrel{{d}} {
\longrightarrow}V, 
\end{eqnarray*}
as shown by \cite{HK96}.

The case $\gamma=1$ of Theorem~\ref{TC2} yields a similar result for
a density
$f(\mathbf{x})= c |\mathbf{x}|^\alpha e^{-\beta|\mathbf{x}|}$.
\end{example}


\begin{example}\label{EGumbelsphere}
Suppose that the points $X_i$ are symmetrically
distributed in the unit sphere with
%
\begin{equation}
\mathbb{P}\bigl(|X|>x\bigr)=e^{-x/(1-x)}= e\cdot e^{-1/(1-x)},\qquad 0\le x<1.
\end{equation}
It is easily verified that (\ref{a}) holds for $t=O(\log\log n)$ with
$a_n=1-\log^{-1}n+\log^{-2}n$ and $b_n=\log^{-2}n$;
cf. \cite{LLR}, Example~1.7.5.
Hence Theorem~\ref{T1} yields a Gumbel limit for $M^{(2)}$ in this
case too.
(For some other distributions in the unit sphere, $M^{(2)}$ has an asymptotic
Weibull distribution as shown by
Mayer and Molchanov \cite{MM07}; see Section~\ref{SSweibull}.)
\end{example}

\section{Proof of Theorem~\texorpdfstring{\protect\ref{T1}}{1.1}}\label{Spf}

Let $\lambda$ be a fixed real number, and define two sequences $r_n$ and
$s_n$ of
positive numbers by
%
\begin{eqnarray}
r_n&:=& \frac{d-1}2\log\frac{a_n}{b_n}-\log\log
\frac
{a_n}{b_n}-\log c_d-\lambda, \label{rn}
\\
s_n&:=& \tfrac{1}2\log r_n. \label{sn}
\end{eqnarray}
(The value of $r_n$ is determined by the argument below,
but $s_n$ could be any sequence that tends to $\infty$ sufficiently slowly.)
Note that $r_n\to\infty$ and $s_n\to\infty$, and $s_n=o(r_n)$;
furthermore, $r_nb_n=o(a_n)$. We assume below tacitly that $n$ is so large
that $r_n>s_n>0$, and $(r_n+s_n)b_n<a_n$.

Further for convenience, we let
%
\begin{equation}
\tau_n:=\frac{d-1}2\log\frac{a_n}{b_n};
\end{equation}
thus (\ref{a}) is assumed to hold for $|t|\le\tau_n$
(and it then automatically holds uniformly for these $t$). Note that
$r_n+s_n\le\tau_n$,
at least for $n$ large; it suffices to consider only such $n$, and
thus (\ref{a}) holds uniformly for $|t|\le r_n+s_n$.

In this section we prove the following result, which immediately implies
Theorem~\ref{T1}, since $W_n$ defined in this theorem is related to $M^{(2)}_n$
by the
relation
$M^{(2)}_n> 2a_n-r_nb_n\iff W_n \neq0$.

\begin{theorem}
\label{T2}
Let $X_1,X_2,\ldots$ be as in Theorem~\ref{T1}, and
let $W_n$ be the number of pairs $(i,j)$ with $1\le i<j\le n$ such that
$|X_i-X_j|>2a_n-r_nb_n$. Then
$W_n\stackrel{{d}}{\longrightarrow}\operatorname
{Po}(e^{-\lambda})$.
\end{theorem}

We shall prove Theorem~\ref{T2} by standard Poisson approximation techniques.
However, some care is needed, since it turns out that the
mean does not converge in Theorem~\ref{T2}; at least in typical cases,
$\mathbb{E}W_n\to\infty$.
The problem is that while~(\ref{max1})--(\ref{max2}) show that the largest
$|X_i|$ typically is about $a_n$, the unlikely event that $\max|X_i|$ is
substantially larger gives a significant contribution to $\mathbb
{E}W_n$, since
an exceptionally large $X_i$ is likely to be part of many pairs with
$|X_i-X_j|>2a_n-r_nb_n$.
(A formal proof can be made by the arguments below, but taking $s_n$ to
be a
large constant times $r_n$.)
We thus do a truncation
(this is where we use $s_n$)
and define, for $\mathbf{x},\mathbf{y}\in\mathbb R^d$, the indicator function
%
\begin{equation}
\label{fn} f_n(\mathbf{x},\mathbf{y}):=\mathbf{1} \bigl\{|
\mathbf{x}-\mathbf {y}|>2a_n-r_nb_n \mbox{
and } |\mathbf{x} |,|\mathbf{y}|\le a_n+s_nb_n
\bigr\}
\end{equation}
{and the corresponding sum}
%
\begin{equation}
\label{w'n} W_n':=\sum_{1\le i<j\le n}f_n(X_i,X_j).
\end{equation}
(This is somewhat similar to the proofs of \cite{MR93} and \cite
{HK96} which
also use Poisson approximation, but they use a decomposition with
several terms.)
Note that if $f_n(\mathbf{x},\mathbf{y})\neq0$, then $|\mathbf
{x}|+|\mathbf{y}|\ge|\mathbf{x}-\mathbf{y}
|>2a_n-r_nb_n$
and thus
%
\begin{equation}
\label{uppner} a_n-(r_n+s_n)b_n
<|\mathbf{x}|,|\mathbf{y}|\le a_n+s_nb_n.
\end{equation}

\begin{remark}\label{Rnot}
The fact that $\mathbb{E}W_n\to\infty$ shows also that the asymptotic
distribution of
$M^{(2)}_n$ is \emph{not} the same as the asymptotic distribution of the
maximum of
${ n\choose2}$ independent random variables with the same distribution as
$|X_1-X_2|$. This is unlike the Weibull case (see Section~\ref
{SSweibull}), where
Mayer and Molchanov \cite{MM07} point out that such an equivalence holds.
\end{remark}

We use the following estimates; proofs are given later in this section.

\begin{lemma} \label{L0}
$\mathbb{P}(W_n\neq W_n')\to0$.
\end{lemma}

\begin{lemma}
\label{LE}
%
\[
\mathbb{E}f_n(X_1,X_2) 
r_ne^{r_n}
=\frac{2+o(1)}{n^2}e^{-\lambda}.
\]
\end{lemma}

\begin{lemma}\label{Lvar}
\[
\mathbb{E} \bigl(f_n(X_1,X_2)f_n(X_1,X_3)
\bigr) = o \bigl(n^{-3} \bigr).
\]
\end{lemma}

\begin{pf*}{Proof of Theorem~\ref{T2} (and thus of Theorem~\ref{T1})}
Consider $W'_n$ defined in (\ref{w'n}), and note that Lemma~\ref{LE}
shows $\mathbb{E}W_n'\to e^{-\lambda}$.
Moreover, the Poisson convergence
%
\begin{equation}
\label{pow'} W_n'\stackrel{{d}} {
\longrightarrow}\operatorname {Po}\bigl(e^{-\lambda}\bigr)
\end{equation}
follows from Lemmas \ref{LE} and~\ref{Lvar} using a theorem
by Silverman and Brown \cite{SB}; see also Barbour and Eagleson \cite{BE84},
Barbour, Holst and Janson \cite{SJI}, Theorem~2.N and Corollary~2.N.1,
and Jammalamadaka and Janson \cite{JJ}, Theorem~3.1 and Remark~3.4.

The conclusion $W_n\stackrel{{d}}{\longrightarrow
}\operatorname{Po}(e^{-\lambda})$ now follows by Lemma~\ref{L0}.
\end{pf*}

\begin{pf*}{Proof of Lemma~\ref{L0}}
We have
\[
\mathbb{P}\bigl(W_n\neq W_n'\bigr) \le
\mathbb{P} \Bigl(\max_{i\le n} |X_i|
>a_n+s_nb_n \Bigr) \le n \mathbb{P}\bigl (|X|
>a_n+s_nb_n \bigr) \to0
\]
by (\ref{a}), since 
$s_n\to\infty$.
\end{pf*}

In order to prove Lemmas \ref{LE} and \ref{Lvar}, we need some estimates.

\begin{lemma}
\label{Lcos}
Let $Y$ and $Z$ be two independent random unit vectors in $\mathbb R^d$
such that
$Y$ is uniformly distributed on the unit sphere $S^{d-1}$, and let
$\Theta$
be the angle between $Y$ and $Z$. Then, as $\varepsilon\searrow0$,
\[
\mathbb{P}(1+\cos\Theta<\varepsilon) \sim \frac{2^{(d-3)/2}\Gamma(d/2)}{\sqrt\pi\Gamma((d+1)/2)} \varepsilon
^{(d-1)/2}.
\]
\end{lemma}

\begin{pf}
By rotational invariance,
we may assume that 
$Z=(1,0,\ldots,0)$. In this case, if
$Y=(\eta_1,\ldots,\eta_d)$, then $\cos\Theta=\langle Y,Z\rangle
=\eta
_1$; moreover,
it is well known (and easily seen) that $\eta_1$ has the density function
\[
g(x)=c''_d\bigl(1-x^2
\bigr)^{(d-3)/2},\qquad -1<x<1,
\]
where
\begin{eqnarray*}
1/c''_d&=& \int_{-1}^1
\bigl(1-x^2\bigr)^{(d-3)/2}\,\mathrm{d}x = \int
_{0}^1(1-y)^{(d-3)/2}y^{-1/2}
\,\mathrm{d}y \\
&=&\frac{\Gamma((d-1)/2)\Gamma(1/2)}{\Gamma(d/2)}.
\end{eqnarray*}
The result follows
by a simple calculation.
\end{pf}

\begin{lemma}
\label{L2}
If\/ $Y$ and $Z$ are two independent
random vectors in $\mathbb R^d$ such that $Y$ is uniformly
distributed on the sphere $|Y|=a_n+tb_n$, and $Z$ has any
distribution on the sphere $|Z|=a_n+ub_n$,
with $|t|,|u|\le r_n+s_n$,
then
uniformly in all such $t$ and $u$,
\[
\mathbb{P}\bigl(|Y-Z|>2a_n-r_nb_n\bigr) \sim
c'''_d \biggl(
\frac{b_n}{a_n} \biggr)^{(d-1)/2}(r_n+t+u)_+^{(d-1)/2}
\]
with
\[
c'''_d:=2^{(d-1)/2}
\frac{2^{(d-3)/2}\Gamma(d/2)}{\sqrt\pi\Gamma((d+1)/2)} =\frac{2^{d-2}\Gamma(d/2)}{\sqrt\pi\Gamma((d+1)/2)}.
\]
%
\end{lemma}

\begin{pf}
By the cosine formula, letting $\Theta$ be the angle between $Y$ and $Z$,
\[
|Y-Z|^2=|Y|^2+|Z|^2-2|Y| |Z|\cos\Theta =
\bigl(|Y|+|Z|\bigr)^2-2|Y| |Z|(1+\cos\Theta).
\]
Hence,
by the assumption $(r_n+s_n)b_n=o(a_n)$ and thus $tb_n,ub_n=o(a_n)$,
\begin{eqnarray*}
&&|Y-Z|>2a_n-r_nb_n
\\
&&\qquad\iff\quad \bigl(2a_n+(t+u)b_n \bigr)^2-2(a_n+tb_n)
(a_n+ub_n) (1+\cos\Theta) \\
&&\qquad\qquad\hspace*{40pt}>(2a_n-r_nb_n)^2
\\
&&\qquad \iff\quad 1+\cos\Theta  <\frac{(2a_n+(t+u)b_n)^2-(2a_n-r_nb_n)^2}{2(a_n+tb_n)(a_n+ub_n)} 
\\
&&\hspace*{40pt}\qquad\qquad\qquad =\frac{2b_n}{a_n}(r_n+t+u) \bigl(1+o(1) \bigr).
\end{eqnarray*}
The result follows by Lemma~\ref{Lcos}
(applied to $Y/|Y|$ and $Z/|Z|$; the angle $\Theta$ remains the same),
using again that $r_nb_n=o(a_n)$;
the probability is obviously 0 when $r_n+t+u\le0$.
\end{pf}

\begin{remark}\label{RL2}
In this section
we use fixed sequences $a_n,b_n,r_n,s_n$,
but we note for future use that Lemma~\ref{L2}
more generally holds for any positive sequences with $(1+r_n+s_n)b_n=o(a_n)$.
\end{remark}

We let $X$ be a random variable with $X\stackrel{{d}}{=}X_i$
and define
%
\begin{equation}
\label{tn} T_n:=\bigl(|X|-a_n\bigr)/b_n;
\end{equation}
thus
$|X| =a_n+T_nb_n$ and (\ref{a}) says that
%
\begin{equation}
\label{tnt} \mathbb{P}(T_n>t)= \frac{1+o(1)}n
e^{-t},
\end{equation}
for all $t=t_n$ with $|t|\le\tau_n$,
and in particular for all $t$ with $|t|\le r_n+s_n$.

\begin{lemma}
\label{L4}
Suppose that the function $h(t)$ is nonnegative,
continuous and increasing in an interval
$[t_0,t_1]$, with $[t_0,t_1]\subseteq[-\tau_n,\tau_n]$.
Then, uniformly for all such intervals $[t_0,t_1]$ and functions $h$,
\[
\mathbb{E} \bigl(h(T_n)\mathbf{1}\{t_0<T_n
\le t_1\} \bigr) =\frac{1+o(1)}n\int_{t_0}^{t_1}h(t)e^{-t}
\,\mathrm{d}t+o \biggl(\frac
{h(t_1)e^{-t_1}}{n} \biggr).
\]
\end{lemma}

\begin{pf}
Let $\mu=\mu_n:={\mathcal L}(T_n)$ denote the distribution of $T_n$.
Then, using (\ref{tnt}) and two integrations by parts,
%
\begin{eqnarray*}
&&\mathbb{E} \bigl(h(T_n)\mathbf{1}\{t_0<T_n
\le t_1\} \bigr)\\
&&\qquad =\int_{t_0+}^{t_1}h(t)
\,\mathrm{d}\mu(t) =-\int_{t_0+}^{t_1}h(t)\,\mathrm{d}
\mathbb{P}(t<T_n\le t_1) 
\\
&&\qquad =h(t_0)\mathbb{P}(t_0< T_n\le
t_1)+\int_{t_0}^{t_1}\,\mathrm{d}h(u)
\mathbb{P}(u<T_n\le t_1) 
\\
&&\qquad =\frac{1+o(1)}n \biggl(h(t_0)e^{-t_0}+\int
_{t_0}^{t_1}\,\mathrm{d}h(u) e^{-u} \biggr)\\
&&\qquad\quad{} -
\frac{1+o(1)}n \biggl(h(t_0)e^{-t_1}+\int
_{t_0}^{t_1}\,\mathrm{d}h(u) e^{-t_1} \biggr)
\\
&&\qquad =\frac{1+o(1)}n\int_{t_0}^{t_1}h(u)e^{-u}
\,\mathrm{d}u +\frac{o(1)}n{h(t_1)e^{-t_1}},
\end{eqnarray*}
%
with all $o(1)$ uniform in $t_0$, $t_1$ and $h$.
\end{pf}

\begin{lemma}
\label{L3}
Let\/ $\mathbf{x}$ be a vector in $\mathbb R^d$ with $|\mathbf
{x}|=a_n+ub_n$ where
$-r_n-s_n< u\le s_n$.
Then, uniformly for all such $\mathbf{x}$,
\[
\mathbb{E}f_n(X,\mathbf{x})=\frac{1+o(1)}nc''''_d
\biggl(\frac
{b_n}{a_n} \biggr)^{(d-1)/2} \bigl(e^{r_n+u}+O
\bigl((r_n+s_n+u+1)^{(d-1)/2}e^{-s_n}
\bigr) \bigr),
\]
where
$c''''_d:=\Gamma((d+1)/2)c'''_d=
2^{d-2}\Gamma(d/2)/\sqrt\pi$.
\end{lemma}

\begin{pf}
We use $T_n$ defined by (\ref{tn}), and note that $f_n(X,\mathbf
{x})=0$ unless
$-r_n-s_n <T_n\le s_n$; see (\ref{uppner}).
Moreover, Lemma~\ref{L2} shows that for $t\in(-r_n-s_n,s_n]$,
\[
\mathbb{E}\bigl(f_n(X,\mathbf{x})\mid T_n=t\bigr) \sim
c'''_d \biggl(
\frac{b_n}{a_n} \biggr)^{(d-1)/2}(r_n+t+u)_+^{(d-1)/2},
\]
uniformly in these $u$ and $t$, and thus
\[
\mathbb{E}f_n(X,\mathbf{x}) \sim c'''_d
\biggl(\frac{b_n}{a_n} \biggr)^{(d-1)/2} \mathbb{E} \bigl((r_n+u+T_n)_+^{(d-1)/2}
\mathbf{1}\{ -r_n-s_n<T_n\le
s_n\} \bigr).
\]

We apply Lemma~\ref{L4} with $h(t)=(r_n+u+t)_+^{(d-1)/2}$ and obtain
\begin{eqnarray*}
&&\mathbb{E} \bigl((r_n+u+T_n)_+^{(d-1)/2}
\mathbf{1}\{ -r_n-s_n<T_n\le
s_n\} \bigr)
\\
&&\qquad =\frac{1+o(1)}n\int_{-r_n-s_n}^{s_n}(r_n+u+t)_+^{(d-1)/2}
e^{-t}\,\mathrm{d}t\\
&&\qquad\quad{} + o \biggl(\frac{(r_n+u+s_n)^{(d-1)/2}e^{-s_n}}{n} \biggr)
\\
&&\qquad =\frac{1+o(1)}n\int_{0}^{r_n+s_n+u}
x^{(d-1)/2} e^{r_n+u-x}\mathrm {d}x + o \biggl(\frac{(r_n+u+s_n)^{(d-1)/2}e^{-s_n}}{n}
\biggr) 
\\
&&\qquad =\frac{1+o(1)}n\Gamma \biggl(\frac{d+1}{2} \biggr)
\bigl(e^{r_n+u}+O \bigl((r_n+s_n+u+1)^{(d-1)/2}e^{-s_n}
\bigr) \bigr),
\end{eqnarray*}
and the result follows.
\end{pf}

\begin{pf*}{Proof of Lemma~\ref{LE}}
We condition on $X_1$ and apply Lemma~\ref{L3}, with $X$ replaced by $X_2$
and $u=T_n$ given by (\ref{tn}) with $X=X_1$; thus, by 
(\ref{uppner}),
%
\begin{eqnarray}
\label{lunsen} &&\mathbb{E} \bigl(f_n(X_1,X_2)
\mid X_1 \bigr)\nonumber\\
&&\qquad = \mathbb{E} \bigl(f_n(X_2,X_1)
\mid X_1 \bigr)
\nonumber
\\[-8pt]
\\[-8pt]
\nonumber
&&\qquad=\frac{1+o(1)}nc''''_d
\biggl(\frac{b_n}{a_n} \biggr)^{(d-1)/2}
\\
&&\qquad\quad{} \times\bigl(e^{r_n+T_n}+O \bigl((r_n+s_n+T_n+1)^{(d-1)/2}e^{-s_n}
\bigr) \bigr) \mathbf{1}\{-r_n-s_n<T_n
\le s_n\}.\nonumber
\end{eqnarray}
Hence
\begin{eqnarray*}
&&\hspace*{-4pt}\mathbb{E}f_n(X_1,X_2)\\
&&\hspace*{-7pt}\qquad= \mathbb{E} \bigl(
\mathbb{E} \bigl(f_n(X_1,X_2)\mid
T_n \bigr) \bigr)
\\
&&\hspace*{-7pt}\qquad=\frac{1+o(1)}nc''''_d
\biggl(\frac{b_n}{a_n} \biggr)^{(d-1)/2}
\\
&&\hspace*{-7pt}\qquad\quad{} \times\mathbb{E} \bigl( \bigl(e^{r_n+T_n}+O \bigl((r_n+s_n+T_n+1)^{(d-1)/2}e^{-s_n}
\bigr) \bigr) \mathbf{1}\{ -r_n-s_n<T_n
\le s_n\} \bigr).
\end{eqnarray*}
By Lemma~\ref{L4} with $h(t)=e^t$ we obtain,
since $s_n=o(r_n)$ and $r_n\to\infty$,
\[
\mathbb{E} \bigl(e^{T_n}\mathbf{1}\{-r_n-s_n<T_n
\le s_n\} \bigr) =\frac{1+o(1)}n(r_n+2s_n)+o
\biggl(\frac{1}{n} \biggr) =\frac{1+o(1)}nr_n
\]
and by Lemma~\ref{L4} with $h(t)=(r_n+s_n+t+1)^{(d-1)/2}$,
\[
\mathbb{E} \bigl((r_n+s_n+T_n+1)^{(d-1)/2}
\mathbf{1}\{-r_n-s_n<T_n\le
s_n\} \bigr) =O \biggl(\frac{e^{r_n+s_n}}{n} \biggr),
\]
and the result follows, using our choice of $r_n$ in (\ref{rn}) and
$c_d=(d-1)c''''_d/4$.
\end{pf*}

\begin{pf*}{Proof of Lemma~\ref{Lvar}}
By (\ref{lunsen}),
\[
\mathbb{E} \bigl(f_n(X_1,X_2)\mid
X_1 \bigr) =O \biggl(\frac{1}n \biggl(\frac{b_n}{a_n}
\biggr)^{(d-1)/2}e^{r_n+T_n}\mathbf{1}\{-r_n-s_n<T_n
\le s_n\} \biggr) ,
\]
where we used $(r_n+s_n+T_n+1)^{(d-1)/2}=O(e^{r_n+s_n+T_n})$.
Hence, using Lemma~\ref{L4} with $h(t)=e^{2t}$,
\begin{eqnarray*}
&&\mathbb{E} \bigl(f_n(X_1,X_2)f_n(X_1,X_3)
\bigr)\\
&&\qquad=\mathbb{E} \bigl(\mathbb{E} \bigl(f_n(X_1,X_2)
\mid X_1 \bigr)^2 \bigr)
\\
&&\qquad =O \biggl(\frac{1}{n^2} \biggl(\frac{b_n}{a_n} \biggr)^{d-1}
\mathbb {E} \bigl(e^{2r_n+2T_n}\mathbf{1}\{-r_n-s_n<T_n
\le s_n\} \bigr) \biggr)
\\
&&\qquad =O \biggl(\frac{1}{n^3} \biggl(\frac{b_n}{a_n} \biggr)^{d-1}e^{2r_n+s_n}
\biggr) =O \biggl(\frac{1}{n^3}\cdot\frac{ e^{s_n}}{r_n^2} \biggr),
\end{eqnarray*}
and the result follows by our choice (\ref{sn}) of $s_n$.
\end{pf*}

\section{Proofs of Theorems \texorpdfstring{\protect\ref{TC1}}{1.4} and
\texorpdfstring{\protect\ref
{TC2}}{1.5}}\label{Spf2}
\mbox{}
\begin{pf*}{Proof of Theorem~\ref{TC1}}
First, by~(\ref{g'}) and $a_n{\to\infty}$, $g'(a_n)>0$ for large $n$ at
least, so
$1/g'(a_n)>0$.
Furthermore, $a_ng'(a_n)\to\infty$ by (\ref{g'}), so
$b_n\sim1/g'(a_n)$ and
$b_n/a_n\to0$ by (\ref{bn}).

We will prove that (\ref{a}) holds, uniformly for all $t$ with
$|t|\le A \log(a_n/b_n)$, for any fixed $A$. The result then follows by
Theorem~\ref{T1}. In order to prove (\ref{a}), we may suppose that
$b_n=1/g'(a_n)$;
the general case (\ref{bn}) follows easily.
We may also suppose that $n$ is large.

Let $A>0$ be a constant, and let, for $x$ so large that $xg'(x)>1$,
%
\begin{equation}
\label{gd} \delta(x):=A\frac{\log (xg'(x) )}{g'(x)}
\end{equation}
and
%
\begin{equation}
I_x:= \bigl[x-\delta(x),x+\delta(x)\bigr].
\end{equation}
Since $\delta(x)/x\to0$ as ${x\to\infty}$ by (\ref{g'}), we may
assume that
$0<\delta(x)<x/2$; hence $I_x\subset(x/2,2x)$.
We claim that, for large $x$,
%
\begin{equation}
\label{ix} \tfrac{1}2g'(x)< g'(y)<
2g'(x),\qquad y\in I_x.
\end{equation}
To show this, assume that (\ref{ix}) fails for some $x$, and let $y$
by the
point in $I_x$ nearest to $x$ where (\ref{ix}) fails.
(If there are two possible choices for $y$, take any of the points.)
Then
%
\begin{equation}
\label{w0} \biggl\llvert \frac{1}{g'(y)}-\frac{1}{g'(x)}\biggr\rrvert
\ge\frac{1}{2g'(x)}.
\end{equation}
By (\ref{w0}) and
the mean value theorem, there exists $z\in[y,x]$ (if $y<x$)
or $z\in[x,y]$ (if $y>x$) such that
%
\begin{equation}
\label{w1} \frac{1}{2g'(x)} \le|y-x| \biggl\llvert \frac{\mathrm{d}}{\mathrm{d}z}
\frac{1}{g'(z)}\biggr\rrvert = |y-x|\frac{|g''(z)| }{g'(z)^2} \le\delta(x)
\frac{|g''(z)| }{g'(z)^2}.
\end{equation}
On the other hand,
$\frac{1}2g'(x)\le g'(z)\le2g'(x)$ by the choice of $y$;
furthermore, $z\in I_x$ so $x/2<z<2x$.
Hence (\ref{g''}) implies, for large $x$, using (\ref{g'}),
%
\begin{equation}
\label{w2} \frac{|g''(z)| }{g'(z)^2} \le\frac{1}{\log^2 (zg'(z) )} \le\frac{2}{\log^2 (xg'(x) )}.
\end{equation}
However, (\ref{w1}) and (\ref{w2}) combine to yield
\[
\frac{1}{2g'(x)} \le\frac{2\delta(x)}{\log^2 (xg'(x) )} = \frac{2A}{g'(x)\log (xg'(x) )},
\]
which contradicts (\ref{g'}) for large $x$.
This contradiction shows that (\ref{ix}) holds for large $x$.

Next, (\ref{gd}), (\ref{ix}) and (\ref{g''}) imply that, for large $x$,
\begin{eqnarray*}
\sup_{y\in I_x}\bigl|g''(y)\bigr|
\delta(x)^2 &=& A^2 \frac{ \sup_{y\in I_x}|g''(y)|}{g'(x)^2}\log^2
\bigl(xg'(x) \bigr)
\\
& \le & 5 A^2 \sup_{y\in I_x}\frac{|g''(y)|}{g'(y)^2}
\log^2 \bigl(yg'(y) \bigr) \to0
\end{eqnarray*}
as ${x\to\infty}$.
Consequently, a Taylor expansion yields, uniformly for $|u|\le\delta(x)$,
%
\begin{equation}
\label{gxu} g(x+u)=g(x)+ug'(x)+o(1)
\end{equation}
as ${x\to\infty}$.
Taking $x=a_n$ and $u=tb_n=t/g'(a_n)$,
with $|t|\le A \log(a_n/b_n)=A\log(a_ng'(a_n))$,
we have $|u|\le\delta(a_n)$ by (\ref{gd}), and thus (\ref{gxu})
applies and
shows, by (\ref{an}) and our choice $b_n=1/g'(a_n)$,
%
\begin{equation}
g(a_n+tb_n)=g(a_n)+tb_ng'(a_n)+o(1)=
\log n + t + o(1),
\end{equation}
uniformly for such $t$.
By (\ref{g}), this yields
\[
\mathbb{P}\bigl(|X_1|>a_n+tb_n\bigr) =\exp \bigl(-
\log n-t+o(1) \bigr)
\]
uniformly for
$|t|\le A \log(a_n/b_n)$, which is (\ref{a}).
The result follows by Theorem~\ref{T1}.
\end{pf*}

Before proving Theorem~\ref{TC2} we give an elementary lemma.

\begin{lemma}\label{Li}
If $\beta,\gamma>0$ and $h(x)$ is a positive differentiable function
such that
$(\log h(x))' = o(x^{\gamma-1})$ as ${x\to\infty}$, then
\[
\label{li} \int_x^\infty h(y)e^{-\beta y^\gamma}
\,\mathrm{d}y \sim(\beta\gamma)^{-1}x^{1-\gamma} h(x)e^{-\beta x^\gamma}
\qquad\mbox{as } {x\to\infty}.
\]
\end{lemma}

\begin{pf}
Suppose first that $\gamma=1$. Then we assume $(\log h(x))'=o(1)$.
Let $\varepsilon(x):=\sup_{y\ge x}|(\log h(y))'|$; thus $\varepsilon
(x)\to0$ as
${x\to\infty}$.
Furthermore, for $t>0$,
\[
\bigl|\log h(x+t)-\log h(x) \bigr| \le\varepsilon(x) t
\]
and thus
\[
e^{-(\beta+\varepsilon(x))t} \le \frac{h(x+t)e^{-\beta(x+t)}}{h(x)e^{-\beta x}} \le e^{-(\beta
-\varepsilon(x))t}.
\]
Integrating we obtain, for $x$ so large that $\varepsilon(x)<\beta$,
\[
\frac{h(x) e^{-\beta x}}{\beta+\varepsilon(x)} \le \int_0^\infty
h(x+t)e^{-\beta(x+t)}\,\mathrm{d}t \le \frac{h(x) e^{-\beta x}}{\beta-\varepsilon(x)} ,
\]
and (\ref{li}) follows when $\gamma=1$.

For a general $\gamma$ we change variable by $y=z^{1/\gamma}$:
\[
\int_x^\infty h(y)e^{-\beta y^\gamma} =\int
_{x^{\gamma}}^\infty h\bigl(x^{1/\gamma}\bigr)
e^{-\beta z} \gamma ^{-1}z^{1/\gamma
-1}\,\mathrm{d}z .
\]
The function $H(z)=\gamma^{-1}h(z^{1/\gamma})z^{1/\gamma-1}$ satisfies
\[
\bigl(\log H(z)\bigr)' =(\log h)'\bigl(z^{1/\gamma}
\bigr) \cdot\gamma^{-1}z^{1/\gamma-1}+\bigl(\gamma ^{-1}-1
\bigr)z^{-1}= o(1),
\]
and thus the case $\gamma=1$ applies and yields
\[
\int_x^\infty h(y)e^{-\beta y^\gamma} = \int
_{x^\gamma}^\infty H(z)e^{-\beta z} \,\mathrm{d}z \sim
\beta^{-1}H\bigl(x^\gamma\bigr)e^{-\beta x^\gamma}\qquad \mbox{as } {x\to
\infty},
\]
which is (\ref{li}).
\end{pf}

\begin{pf*}{Proof of Theorem~\ref{TC2}}
Let $\omega_{d-1}:=2\pi^{d/2}/\Gamma(d/2)$, the surface area of the
unit sphere $\mathbb S^{d-1}$ in $\mathbb R^d$.
By (\ref{f}) and
Lemma~\ref{Li}, with $h(x)=x^{\alpha+d-1}$,
\[
\mathbb{P}\bigl(|X_1|>x\bigr) \sim\int_x^\infty
c r^\alpha e^{-\beta r^\gamma} \omega _{d-1}r^{d-1}
\,\mathrm{d}r \sim c\omega_{d-1}(\beta\gamma)^{-1}x^{\alpha+d-\gamma}e^{-\beta
x^\gamma}.
\]
Hence (\ref{g}) holds with
%
\begin{equation}
\label{g0} g(x)=\beta x^\gamma-(\alpha+d-\gamma)\log x + \log(\beta
\gamma /c\omega_{d-1}).
\end{equation}
We have
%
\begin{eqnarray}
\label{g'0} g'(x)&=&\beta\gamma x^{\gamma-1} -(\alpha+d-
\gamma)x^{-1},
\\
g''(x)&=&\beta\gamma(\gamma-1) x^{\gamma-2} +(
\alpha+d-\gamma)x^{-2}, \label{g''0}
\end{eqnarray}
and (\ref{g'})--(\ref{g''}) are easily verified.

In order to have (\ref{an}), we need,
since $g(x)\sim\beta x^\gamma$ as ${x\to\infty}$ by (\ref{g0}),
$a_n\sim\beta^{-1/\gamma}\log^{1/\gamma}n$; furthermore, (\ref
{bn}) then yields
\[
b_n\sim\frac{1}{g'(a_n)} \sim\frac{1}{\beta\gamma a_n^{\gamma-1}} \sim
\beta^{-1/\gamma}\gamma^{-1}\log^{1/\gamma-1}n.
\]
We thus choose, for simplicity,
%
\begin{equation}
\label{bn0} b_n:= \beta^{-1/\gamma}\gamma^{-1}
\log^{1/\gamma-1}n.
\end{equation}

If $u_n=O(\log\log n)$, then, by a Taylor expansion and
(\ref{g0})--(\ref{g''0}),
\begin{eqnarray*}
&& g \bigl(\beta^{-1/\gamma}\log^{1/\gamma}n +
u_nb_n \bigr)
\\
&&\qquad = g \bigl(\beta^{-1/\gamma}\log^{1/\gamma}n \bigr) +
u_nb_n g' \bigl(\beta^{-1/\gamma}
\log^{1/\gamma}n \bigr) +o(1)
\\
&&\qquad =\log n -(\alpha+d-\gamma)\gamma^{-1}(\log\log n-\log\beta) +\log(
\beta\gamma/c\omega_{d-1})
\\
&&\qquad\quad{} +u_n +o(1).
\end{eqnarray*}
Hence we define
%
\begin{eqnarray}
\label{an0} 
&&a_n := \beta^{-1/\gamma}
\log^{1/\gamma}n + 
b_n \biggl(\frac{\alpha+d-\gamma}{\gamma}\log
\log n - (\alpha+d)\gamma^{-1}\log\beta
\nonumber
\\[-8pt]
\\[-8pt]
\nonumber
&&\hspace*{199pt}{}-\log\gamma+\log(c\omega_{d-1}) \biggr) 
\end{eqnarray}
and find that (\ref{an}) holds. 
Furthermore, by another Taylor expansion,
\[
g'(a_n)=\beta^{1/\gamma}\gamma\log^{1-1/\gamma}n
\cdot \bigl(1+O \bigl(\log^{-1}n+\log\log n\cdot\log^{-1/\gamma
}n
\bigr) \bigr),
\]
and (\ref{bn}) follows easily.
Hence Theorem~\ref{TC1} applies, and (\ref{t1a}) holds.
Moreover, by (\ref{an0}) and (\ref{bn0}),
\begin{eqnarray*}
a_n/b_n &\sim&\gamma\log n,
\\
\log(a_n/b_n) &=& \log\log n+\log\gamma+ o(1),
\\
\log\log(a_n/b_n) &=& \log\log\log n + o(1),
\end{eqnarray*}
and the result follows from (\ref{t1a}) by collecting terms, with
$c'_d=c_d\omega_{d-1}^2$, which yields (\ref{cct}).
\end{pf*}

\section{Further comments}\label{Sfurther}

\subsection{Weibull-type extremes}\label{SSweibull}

The Weibull-type extreme value distribution occurs for random variables that
are bounded above; in our context this means that $|X|$ is bounded, so
$X$ takes values in a bounded set.
(However, there are examples of Gumbel-type in this case too; see
Example~\ref{EGumbelsphere}.)
By scaling we may assume that the upper
endpoint of the support of $|X|$ is 1, so $X$ belongs to the unit ball, but
not always to any smaller ball. The typical Weibull case is
%
\begin{equation}
\label{wei1} \mathbb{P}\bigl(|X|>x\bigr)\sim c(1-x)^\alpha,\qquad x\nearrow1,
\end{equation}
for some $\alpha>0$, in which case
%
\begin{equation}
\label{weiM} \mathbb{P} \bigl(n^{1/\alpha} (1-M_n)>x \bigr)\to
\exp \bigl(-cx^\alpha \bigr),
\end{equation}
which means that $c^{1/\alpha}n^{1/\alpha}(M_n-1)$ converges to the
(negative) Weibull
distribution in (\ref{weibull}).

Mayer and Molchanov \cite{MM07} show that if (\ref{wei1}) holds, then
$M^{(2)}_n$ also has an asymptotic Weibull distribution, with a different
parameter.
More precisely, they show the following;
see also Lao and Mayer \cite{LaoMayer} and Lao \cite{Lao}, which
contain further related
results.

\begin{theorem}[(Mayer and Molchanov \cite{MM07})]\label{Tweibull}
Suppose that $d\ge2$ and that
$X_1,X_2,\ldots$ are i.i.d. $\mathbb R^d$-valued random vectors
with a spherically symmetric distribution
such that (\ref{wei1}) hold for some $\alpha\ge0$ and $c>0$. Then
%
\begin{equation}
\label{weiMM} \mathbb{P} \bigl(n^{4/(d-1+4\alpha)}\bigl(2-M^{(2)}_n
\bigr)>x \bigr) \to \exp \bigl(-c_{d,\alpha}c^2 x^{(d-1+4\alpha)/2}
\bigr),
\end{equation}
with
%
\begin{equation}
c_{d,\alpha}:=\frac{\Gamma(\alpha+1)^2\Gamma((d+1)/2)}{2\Gamma
((d+1+4\alpha)/2)}c'''_d
=\frac{2^{d-3}\Gamma(\alpha+1)^2\Gamma(d/2)}{\sqrt\pi\Gamma
((d+1+4\alpha)/2)}.
\end{equation}
\end{theorem}

Hence, $n^{4/(d-1+4\alpha)}(M^{(2)}_n-2)$ has, apart from a constant factor,
the (negative) Weibull distribution (\ref{weibull}) with parameter
$(d-1+4\alpha)/2$.

Note that Theorem~\ref{Tweibull} includes the case
$\alpha=0$, that is, when $\mathbb{P}(|X|=1)=c>0$, in particular the case
$|X|=1$ with $X$ uniformly distributed on the unit sphere.
(The latter case was earlier shown by Appel and Russo \cite{AR06}.)
In the case
$\alpha=0$, (\ref{weiM}) does not make sense; the asymptotic
distribution of
$M_n$ is degenerate, since $\mathbb{P}(M_n=1)\to1$.

Theorem~\ref{Tweibull} can easily be proved by the method in
Section~\ref{Spf},
taking $a_n:=1$, $b_n:=c^{-1/\alpha}n^{-1/\alpha}$ (with $b_n:=1$
when $\alpha=0$),
$r_n:=x b_n^{-1}n^{-4/(d-1+4\alpha)}$ and $s_n=0$.
(We take $s_n=0$ since no truncation is needed in this case; indeed,
$W_n'=W_n$; cf. Remark~\ref{Rnot}.)
We omit the details. (The authors of \cite{MM07} and \cite{AR06}
also use Poisson approximation, but the details are
different.) 

%

\begin{remark}
Since the normalizing factors in (\ref{weiM}) and (\ref{weiMM}) have
different powers of $n$, $2-M^{(2)}_n$ is asymptotically much larger than
$1-M_n$, and thus $2M_n-M^{(2)}_n$ has the same asymptotic distribution as
$2-M^{(2)}_n$; see (\ref{weiMM}), and cf. Remark~\ref{Rdiff} for the
Gumbel case.
\end{remark}

We have here for simplicity
considered only the standard case when (\ref{wei1}) holds,
and leave extensions to more general distributions with $M_n$ asymptotically
Weibull to the reader.

\subsection{Fr\'echet-type extremes}\label{SSfrechet}

The Fr\'echet-type extreme value distribution occurs for $|X|$ if
(and only
if, see \cite{LLR}, Theorem~1.6.2 and Corollary~1.6.3) there exists a
sequence $\gamma_n\to\infty$ such that
%
\begin{equation}
\label{fre1} \mathbb{P}\bigl(|X|>x\gamma_n\bigr)\sim\frac{1}n
x^{-\alpha}
\end{equation}
for every (fixed) $x>0$; then
%
\begin{equation}
\label{freM} \gamma_n^{-1}M_n \stackrel{
{d}} {\longrightarrow}\tilde V,
\end{equation}
where $\tilde V$ has the Fr\'echet distribution (\ref{frechet}).
The typical case is a power-law tail
%
\begin{equation}
\label{fre2} \mathbb{P}\bigl(|X|>x\bigr)\sim c x^{-\alpha}\qquad \mbox{as ${x\to
\infty}$};
\end{equation}
in this case, (\ref{fre1}) and (\ref{freM}) hold with $\gamma
_n=(cn)^{1/\alpha}$.
We have the following result, independently found by Henze and Lao
\cite{HenzeLao}.
Let again $\omega_{d-1}:=2\pi^{d/2}/\Gamma(d/2)$, the surface area
of the
unit sphere in $\mathbb R^d$.

\begin{theorem}[(Henze and Lao \cite{HenzeLao})]\label{Tfrechet}
Suppose that $d\ge2$ and that
$X_1,X_2,\ldots$ are i.i.d. $\mathbb R^d$-valued random vectors
with a spherically symmetric distribution
such that (\ref{fre1}) hold for some $\gamma_n\to\infty$. Then
%
\begin{equation}
\label{freMM} \gamma_n^{-1}M^{(2)}_n\stackrel{{d}} {\longrightarrow }Z_\alpha
\end{equation}
for some random variable $Z_\alpha$, which can be described as the maximum
distance $\max_{i,j}|\xi_i-\xi_j|$
between the points in a Poisson point process $\Xi=\{\xi_i\}$ on
$\mathbb R^d\setminus\{0\}$ with intensity
$\alpha\omega_{d-1}^{-1}|\mathbf{x}|^{-\alpha-d}$.
\end{theorem}

\begin{pf*}{Sketch of proof}
It is easy to see that the scaled set of points $\{\gamma_n^{-1}
X_i\dvtx\break  1\le i\le n\}$, regarded as a point process on $\mathbb
R^d\setminus0$, converges in
distribution to the Poisson process $\Xi$. It then follows that the
maximum interpoint distance converges. We omit the details.
See, for example,
Kallenberg \cite{Kallenberg:RM} or \cite{Kallenberg} for details on
point processes, or
Janson \cite{SJ136}, Section~4, for a brief summary.
\end{pf*}

Note that the point process $\Xi$ has infinite intensity, and thus a.s. an
infinite number of points, clustering at 0, but a.s. only a finite number
of points $|\xi|>\varepsilon$ for any $\varepsilon>0$. (This is the
reason for
regarding the
point processes on $\mathbb R^d\setminus0$ only,
since we want the point processes to
be locally finite.)

We leave it as an open problem to find an explicit description of the
limit distribution, that is, the distribution of $Z_\alpha$. We do not believe
that it is Fr\'echet, so $M^{(2)}_n$ and $M_n$ will (presumably)
not have the same type of
asymptotic distribution in the Fr\'echet case, unlike the Gumbel and
Weibull cases treated above.

One reason for the more complicated limit behavior in the Fr\'echet
case is
that the Poisson approximation argument in Section~\ref{Spf} fails. If
we define
$W_n'$ as there, with a suitable threshold and a suitable truncation
(avoiding small $|X_i|$ this time), we can achieve $\mathbb
{E}f_n(X_1,X_2)\sim C
n^{-2}$ as in Lemma~\ref{LE}, for a constant $C>0$, but then $\mathbb{E}
f_n(X_1,X_2)f_n(X_1,X_3)$ will be of order $n^{-3}$ and there is no analogue
of Lemma~\ref{Lvar}; this ought to mean that $W_n$ does not have an asymptotic
Poisson distribution. In other words, the problem is that there is too much
dependence between pairs with a large distance.

Furthermore, it can be seen from Theorem~\ref{Tfrechet} that there is
a positive
limiting probability that the maximum distance $M^{(2)}_n$ is attained between
the two vectors $X_{(1)}$ and $X_{(2)}$ with largest length, but it can also
be attained (with probability bounded away from 0)
by any other pair $X_{(k)}$ and $X_{(l)}$ with given $1\le k<l$.
This is related to the preceding comment, and may thus also be
a reason for the more complicated behavior of $M^{(2)}_n$ in the Fr\'echet
case.
(It shows also that there is an asymptotic dependence between
$M^{(2)}_n$ and
$M_n$ which does not exist in the Gumbel and Weibull cases.)
Moreover,
it follows also
that, again unlike the Gumbel and Weibull cases,
the angle between the maximizing vectors $X_i$
and $X_j$ is not necessarily close to $\pi$
(it can be any angle $>\pi/3$), so it is not enough to use
asymptotic estimates as Lemma~\ref{Lcos}.

\subsection{The case $d=1$}\label{SSd=1}
The theorem above supposes $d>1$, for example, because we need
$r_n\to\infty$.
In the case $d=1$, there is a similar result, which is much simpler, but
somewhat different; for comparison we give this result too.
(For simplicity we continue to consider symmetric variables.)

\begin{theorem}\label{T1d=1}
Suppose that $X_1,X_2,\ldots$ are i.i.d. symmetric real-valued
random variables
such that for some sequences $a_n$ and $b_n$ of positive numbers,
(\ref{a})~holds
as ${n\to\infty}$,
for any fixed real $t$.
Then
%
\begin{equation}
\label{t1a1} \frac{M^{(2)}_n-2a_n}{b_n} +2\log2 \stackrel{{d}} {
\longrightarrow}V_++V_-,
\end{equation}
where $V_1,V_2$ are two independent random variables
with the Gumbel distribution $\mathbb{P}(V_\pm\le t)=e^{-e^{-t}}$.
\end{theorem}

\begin{pf}
When $d=1$,
%
\begin{equation}
\label{nasten} M^{(2)}_n=M^+_n-M^-_n,
\end{equation}
where
$M^+_n:=\max_{i\le n} X_i$ and
$M^-_n:=\min_{i\le n} X_i$.

There are about $n/2$ positive and $n/2$ negative $X_i$.
More precisely, denoting these numbers by $N_+$ and $N_-=n-N_+$,
where we assign a random sign also to any value that is 0, we have
$N_+,N_-\sim\operatorname{Bi}(n,1/2)$.
Conditioned on $N_+$ and $N_-$, and assuming that both are nonzero,
$M^+_n$ and $M^-_n$ are independent, with $M^+_n\stackrel{
{d}}{=}M_{N_+}$ and
$M^-_n\stackrel{{d}}{=}-M_{N_-}$.
Moreover, if we assume (\ref{a}) for every fixed $t$, then it is easy
to see that for any random $N$ with $N/n\stackrel{
{p}}{\longrightarrow}1/2$ as ${n\to\infty}$, we have
%
\begin{equation}
\frac{M_N-a_n}{b_n}+\log2\stackrel{{d}} {\longrightarrow}V;
\end{equation}
cf. (\ref{max2}). Consequently,
%
\begin{equation}
\label{vilan} \frac{\pm M^\pm_n-a_n}{b_n}+\log2\stackrel{ {d}} {
\longrightarrow}V_\pm,
\end{equation}
where $V_\pm$ are two random variable{s} with the same Gumbel distribution;
moreover, it is easy to see that this holds jointly with $V_+$ and $V_-$
independently.
The result follows from (\ref{nasten}) and (\ref{vilan}).
\end{pf}

Comparing Theorem~\ref{T1d=1} to Theorem~\ref{T1}, we see that
first of all the limit distribution is different.
Furthermore, the term $\frac{d-1}2\log(a_n/b_n)$ in (\ref{t1a}) disappears,
which is expected, but also the
term $\log\log(a_n/b_n)$ disappears, and the constant term is different,
with $-\log c_d$ replaced by $2\log2$. [$c_d$ in (\ref{cd}) would be
0 for
$d=1$, which does not make sense in (\ref{t1a}).]
A reason for the different behavior is that for $d=1$, there is no issue
with the angles, and $M^{(2)}$ is the sum of two extreme value statistics
($M_n^+$ and $-M_n^-$ in the proof above).

Similarly, in the special case in Theorem~\ref{TC2}, we obtain for $d=1$
from (\ref{an0}) and
(\ref{bn0}) (which hold also for $d=1$ by the proof above) the following,
where the limit distribution again is different; furthermore, the
$\log\log\log n$ term disappears, and the constant term is slightly
different.

\begin{theorem}
\label{TC2d=1}
Suppose that $X_1,X_2,\ldots$ are i.i.d. symmetric real-valued
random variables
with a density function $f(x)$ such
that, as $|x|\to\infty$,
%
\begin{equation}
f(x)\sim c |x|^\alpha e^{-\beta|x|^\gamma}
\end{equation}
for some $c,\beta,\gamma>0$ and $\alpha\in\mathbb R$.
Then
\begin{eqnarray*}
&&\bigl(\beta^{1/\gamma}\gamma\log^{1-1/\gamma}n \bigr)\cdot
M^{(2)}_n - \biggl( 2\gamma\log n + \biggl(2
\frac{\alpha+1}{\gamma}-2 \biggr)\log\log n
\\
&&\hspace*{163pt}\quad{}+\log \bigl(\beta^{-2(\alpha+1)/\gamma}\gamma^{-2}c^2 \bigr)
\biggr) \\
&&\qquad\stackrel{{d}} {\longrightarrow}V_++V_-,
\end{eqnarray*}
where $V_\pm$ are independent and have the Gumbel distribution.
\qed
\end{theorem}

Typical examples are given by $f(x)=(2\pi)^{-1/2}e^{-x^2/2}$ and
$f(x)=\frac{1}2e^{-|x|}$; we leave the details to the reader.

The argument above applies also to the Weibull and Fr\'echet cases when
$d=1$; we omit the details.
(Furthermore, Theorem~\ref{Tfrechet} holds also for $d=1$.)

\subsection{Nonsymmetric distributions}
We have assumed that the distribution of $X$ is spherically symmetric.
What happens if we relax that condition? Consider, for example, the
case of a
normal distribution with a nonisotropic covariance matrix, for
example, with
a simple largest eigenvalue so that there is a unique direction where the
variance is largest. Will the asymptotic distribution of $M^{(2)}_n$
then be
governed mainly by the component in that direction only, so that there
is a
limit law similar to the case $d=1$, or will the result still be
similar to
the theorems above for the spherically symmetric case, or is the result
somewhere in between? We leave this as an open problem.

For the case of points distributed inside a bounded set,
Appel, Najim and Russo~\cite{ANR02}, Mayer and Molchanov \cite{MM07},
Lao and Mayer \cite{LaoMayer} and Lao \cite{Lao}
have results also in the nonsymmetric case.
As an example, consider points uniformly distributed inside an ellips with
major axis 1 and minor axis $b<1$. The maximum distance is obviously
attained by some pair of points close to the endpoints of the major axis,
and it can be shown, by arguments similar to the proof of Theorem~\ref
{Tfrechet},
that
$n^{2/3}(2-M^{(2)}_n)\stackrel{{d}}{\longrightarrow}Z$, where
$Z$ can be described as
the distribution of
$\pi^{2/3}\min_{i,j} (x_i'+x_j''-b^2(y_i'-y_j'')^2/4 )$,
with $\{(x_i',y_i')\}$ and $\{(x_j'',y_j'')\}$,
two independent Poisson
processes with intensity 1 in the parabola $\{(x,y)\dvtx y^2\le2x\}$.
[If the endpoints of the major axis are $(\pm1,0)$, we represent
points close
to them as $(1-x',by')$ and $(-1+x'',by'')$, and note that the distance
$ |(1-x',by')-(-1+x'',by'') |\approx2-x'-x''+b^2(y'-y'')^2/4$;
we omit
the details.] We do not know any explicit description of this limit
distribution.
It seems likely that limits of similar types arise
also in other cases where
$\max|X_i|$ is attained in a single direction,
for example, a 3-dimensional ellipsoid with semiaxes $a>b\ge c $,
while we believe that there is a Weibull limit similar to Theorem~\ref
{Tweibull} if
$a=b>c$, so that there is rotational symmetry around the shortest axis.

\subsection{Other norms}
We have considered here only the Euclidean distance. It seems to be an open
problem to find similar results for other distances, for example, the
$\ell^1$-norm or the $\ell^\infty$-norm in $\mathbb R^d$.

\section*{Acknowledgments}
We wish to thank Professor Norbert Henze for bringing this problem to the
attention of one of us and two anonymous referees for very helpful comments.

This paper was largely written on the occasion of SRJ's visit to Uppsala
in October 2012
to receive an honorary doctorate from
the Swedish University of Agricultural Sciences.





%




\printaddresses
\end{document}